\documentclass[12pt]{amsart}
\usepackage{latexsym}
\usepackage{amsfonts,amssymb,amsmath,amsthm,amscd}
\usepackage[mathscr]{eucal}
\usepackage{graphicx}
\usepackage{xcolor}

\newcommand{\so}{\mbox{${\mathfrak s \mathfrak o}$}}

\newcommand{\m}{\mbox{${\mathfrak m}$}}

\newcommand{\C}{\mbox{${\mathbb C}$}}

\newcommand{\I}{\mbox{${\mathbb I}$}}

\newcommand{\R}{\mbox{${\mathbb R}$}}
\newcommand{\D}{\mbox{${\mathbb S}$}}

\newcommand{\tr}{{\rm tr}}

\newcommand{\End}{{\rm End}}

\newcommand{\Ric}{{\rm Ric}}

\newcommand{\SO}{{\rm SO}}
\newcommand{\SU}{{\rm SU}}
\newcommand{\Sp}{{\rm Sp}}
\newcommand{\Spin}{{\rm Spin}}
\newcommand{\U}{{\rm U}}

\newcommand{\scal}{{\rm scal}}
\newcommand{\T}{{\rm T}}
\newcommand{\Sym}{{\rm Sym}}
\newcommand{\id}{\rm id}
\def\bea{\begin{eqnarray*}}
\def\eea{\end{eqnarray*}}

\newcommand{\eop}{\mbox{$\Box$}}

\makeatletter
\def\numberwithin#1#2{\@ifundefined{c@#1}{\@nocnterrr}{%
  \@ifundefined{c@#2}{\@nocnterr}{%
  \@addtoreset{#1}{#2}%
  \toks@\expandafter\expandafter\expandafter{\csname the#1\endcsname}%
  \expandafter\xdef\csname the#1\endcsname
    {\expandafter\noexpand\csname the#2\endcsname
     .\the\toks@}}}}
\makeatother
\numberwithin{equation}{section}

\newtheorem{thm}[equation]{Theorem}
\newtheorem{lemma}[equation]{Lemma}
\newtheorem{prop}[equation]{Proposition}

\newtheorem{cor}[equation]{Corollary}

\newtheorem{ex}[equation]{Example}

\newtheorem{rem}[equation]{Remark}
\newenvironment{rmk}{\begin{rem} \em}{\end{rem}}

\def\Hom{\mathrm{Hom}}
\def\sp{\mathfrak{sp}}
\def\Cas{\mathrm{Cas}}

\begin{document}

\title[Linear Instability of Sasaki Einstein and nearly parallel ${\rm G}_2$ manifolds]{Linear Instability of Sasaki Einstein and \\ nearly parallel ${\rm G}_2$ manifolds}
\author{Uwe Semmelmann}
\address{Institut f\"ur Geometrie und Topologie \\
Fachbereich Mathematik\\
Universit{\"a}t Stuttgart\\
Pfaffenwaldring 57 \\
70569 Stuttgart, Germany}
\email{uwe.semmelmann@mathematik.uni-stuttgart.de}

\author{Changliang Wang}
\address{School of Mathematical Sciences and Institute for Advanced Study, Tongji University, Shanghai 200092, China}
\email{wangchl@tongji.edu.cn}

\author{M. Y.-K. Wang}
\address{Department of Mathematics and Statistics, McMaster
University, Hamilton, Ontario, L8S 4K1, CANADA}
\email{wang@mcmaster.ca}

\date{revised \today}

\begin{abstract}
{In this article we study the stability problem for the Einstein metrics on Sasaki Einstein and on complete nearly parallel ${\rm G}_2$ manifolds. In the Sasaki case we show linear instability if the second Betti number is positive. Similarly we prove that nearly parallel $\rm G_2$ manifolds with positive third Betti number are
linearly unstable. Moreover, we prove linear instability for the Berger space $\SO(5)/\SO(3)_{irr} $ which is a $7$-dimensional homology sphere with a proper nearly parallel ${\rm G}_2$ structure.}
\end{abstract}

\maketitle



\medskip
\setcounter{section}{0}


\section{\bf Introduction}

In this article we continue the investigation of the linear instability of Einstein manifolds admitting a non-trivial
real Killing spinor. Recall that the linear stability of complete Einstein manifolds admitting a non-trivial parallel
or imaginary Killing spinor has been established in \cite{DWW05}, \cite{Kr17}, and \cite{Wan17}. It is therefore
of some interest to consider the stability problem for complete Einstein manifolds which admit a real Killing spinor,
especially because these manifolds admit more geometric structure than generic Einstein manifolds with positive scalar curvature, and in view of their role in supersymmetric grand unification theories in physics over the years.

We refer the reader to \cite{WW18} for a summary of the various notions of stability under consideration and for a
description of different cases of the general problem. Here we only note that in this paper linear instability refers
to the second variation at Einstein metrics for the Einstein-Hilbert action.  Explicitly, this means that there
exists a non-trivial symmetric $2$-tensor $h$ that is {\em transverse traceless}  (``TT" in short),
i.e., $\tr_g h = 0, \delta_g h = 0$,  such that
\begin{equation}  \label{instability}
\langle \nabla^* \nabla h - 2 \mathring{R} h, h \rangle_{L^2(M, g)} = \langle (\Delta_L - 2E)h, h \rangle_{L^2(M, g)} < 0,
\end{equation}
where $E$ is the Einstein constant, $\Delta_L$ is the positive Lichnerowicz Laplacian, $\mathring{R}h$
is the action of the curvature tensor on symmetric $2$-tensors, and $\nabla^* \nabla$ is the (positive) rough Laplacian.
Condition (\ref{instability}) implies linear instability with respect to Perelman's $\nu$-entropy as well as
dynamical instability with respect to the Ricci flow (by a theorem of Kr\"oncke \cite{Kr15}).

Recall that complete spin manifolds with constant positive sectional curvature are stable Einstein manifolds. They
are exceptional in the sense that they also admit a maximal family of real Killing spinors. For the sake of a smoother
exposition we will henceforth exclude these manifolds from discussion. It then follows that the only even dimension
for which there are complete metrics admitting a non-trivial real Killing spinor is six. In this
situation the Einstein manifolds are strict nearly K\"ahler or else isometric to round $S^6$. (Recall that a strict, nearly
K\"ahler manifold is an almost Hermitian manifold for which the almost complex structure $J$ is non-parallel and
satisfies $(\nabla_X J)X = 0$, where $X$ is any tangent vector and $\nabla$ is the Levi-Civita connection.)  The round sphere
is clearly stable, but it is distinguished by having a maximal family of real Killing spinors.  For the first case
we showed in \cite{SWW20} that if either the second or third Betti number of the manifold is nonzero, then the nearly
K\"ahler metric is linearly unstable. A topological consequence of this fact is that a complete, strict, nearly
K\"ahler $6$-manifold that is linearly stable must be a rational homology sphere.

In this paper, we first consider the Sasaki Einstein case, which arises in all odd dimensions. To simplify matters,
we will take a  Sasaki Einstein manifold to be an odd-dimensional Einstein Riemannian manifold $(M^{2n+1}, g)$
together with
\begin{enumerate}
\item[(a)] a contact $1$-form $\eta$ whose dual vector field $\xi$ is a unit-length Killing field, (i.e., $\eta \wedge (d\eta)^n \neq 0 $ 
everywhere on $M$, $\eta(\xi)=1$, and $L_{\xi} g = 0$), ($K$-contact condition)
\item[(b)] an endomorphism $\Phi: TM \rightarrow TM$ satisfying, for all tangent vectors $X$ and $Y$, the equations
$$ \Phi^2 = - \I + \eta \otimes \xi, \,\,\,\,\, g(\Phi(X), \Phi(Y)) = g(X, Y) - \eta(X) \eta(Y), \,\, \mbox{\rm and}$$
\item[(c)] $d\eta(X, Y) = 2 g(X, \Phi(Y)).$
\end{enumerate}
Systematic expositions of Sasaki Einstein manifolds can be found in \cite{BFGK91} and especially in \cite{BG08}.
An equivalent characterization of a Sasaki Einstein manifold is an Einstein manifold whose metric cone has holonomy
lying in $\SU(n+1)$ \cite{Ba93}. Our first result is:

\begin{thm}  \label{SE}
A complete Sasaki Einstein manifold of dimension $>3$ with non-zero second Betti number $b_2$ must
be linearly unstable with respect to the Einstein-Hilbert action and hence dynamically unstable for the
Ricci flow. More precisely, the coindex of such an Einstein metric, i.e., the dimension of the maximal
subspace of the space of transverse traceless symmetric $2$-tensors on which $($\ref{instability}$)$ holds, is $\geq b_2$.
\end{thm}

This result may be viewed as an odd-dimensional analogue of the observation by Cao-Hamilton-Ilmanen \cite{CHI04} that
a Fano K\"ahler Einstein manifold with second Betti number $> 1$ is linearly unstable. Note that our result applies
to irregular Sasakian Einstein manifolds as well. For dimension $5$, it follows immediately that if a
complete Sasaki Einstein $5$-manifold is linearly stable then it must be a rational homology sphere.
One can also apply the above result to any of the $2$-sphere's worth of compatible Sasaki Einstein structures
on a $3$-Sasakian manifold with nonzero second Betti number to deduce their linear instability. Recall that
 $3$-Sasakian manifolds are automatically Einstein.

In dimension seven, a complete simply connected Riemannian manifold admitting a non-trivial real Killing spinor
is called a {\it nearly parallel $\mathrm G_2$ manifold}.  As we have excluded round spheres,  such a manifold falls
into one of three classes depending on whether the dimension of the space
of Killing spinors is $3, 2$ or $1$. These are respectively the $3$-Sasakian, Sasaki Einstein but not $3$-Sasaki,
and the proper nearly parallel $\rm{G}_2$ cases.  For this family we obtained the following

\begin{thm} \label{b3}
Let $(M^7, g)$ be a complete nearly parallel ${\rm G}_2$-manifold.  Then the coindex of the Einstein metric $g$ is at least
$b_3$. If the manifold is in addition Sasaki Einstein then the coindex is at least $b_2 + b_3$. Hence in the latter
case if such an manifold is linearly stable, then it must be a rational homology sphere.
\end{thm}

While there are numerous examples of complete Sasakian Einstein manifolds with nonzero second Betti number \cite{BG08}
there are relatively fewer examples with $b_3 \neq 0$. To our knowledge they include fourteen examples due
to C. Boyer \cite{Bo08} and ten recent examples by R. G. Gomez \cite{Go19}, all occurring in dimension $7$.
In Gomez's examples, $10 \leq b_3 \leq 20$. Note, however, that for complete $3$-Sasakian manifolds
Galicki and Salamon \cite{GaS96} showed that all their odd Betti numbers must vanish.
We also do not know any example of a {\it proper} nearly parallel ${\mathrm G}_2$ manifold with nonzero third Betti
number. Settling the existence question for such an example would be of interest.

We turn next to the special case of a simply connected closed Einstein $7$-manifold that is homogeneous with respect
to the isometric action of some semisimple Lie group. In \cite{WW18} it was shown that if the manifold is not
locally symmetric then it must be linearly unstable with the possible exception of the isotropy irreducible space
$\Sp(2)/\Sp(1)\approx \SO(5)/\SO(3)$. Here the embedding of $\Sp(1)$ is given by the irreducible complex $4$-dimensional
representation (which is symplectic), while the embedding of $\SO(3)$ is given by the irreducible complex $5$-dimensional
representation (which is orthogonal).  The Einstein metric in this unresolved case is known to be of proper nearly parallel $\rm{G}_2$ type.

\begin{thm} \label{Berger}
The isotropy irreducible Berger space $\Sp(2)/\Sp(1)_{irr} \approx \SO(5)/\SO(3)_{irr} $ is linearly unstable.
\end{thm}

Recall that the homogeneous Einstein metrics (two up to isometry) on the Aloff-Wallach manifolds
$N_{k, l} = \SU(3)/T_{kl}$, where $T_{kl}$ is a closed circle subgroup,
are also of proper nearly parallel $\rm{G}_2$ type, with the exception of
one of the Einstein metrics on $N_{1,1}$, which is $3$-Sasakian. It was shown in \cite{WW18} that all
these Einstein metrics are linearly unstable. One therefore deduces the conclusion that all compact
simply connected homogeneous Einstein $7$-manifolds (many of them admitting a non-trivial real Killing spinor)
are linearly unstable.

Interestingly, the Berger space is a rational homology sphere, so the methods based on constructing destabilizing
directions via harmonic forms do not apply. Likewise, the Stiefel manifold $\SO(5)/\SO(3)$ (the embedding of
$\SO(3)$ is the usual $3$-dimensional vector representation) is also a homology $7$-sphere. The Einstein metric
here is instead of regular Sasakian Einstein type (the Stiefel manifold is a circle bundle over  the Hermitian
symmetric Grassmanian $\SO(5)/S({\rm O}(3){\rm O}(2))$. It was shown to be linearly unstable in \cite{WW18}
by examining the scalar curvature function for homogeneous metrics. But there are also non-homogeneous
$\nu$-unstable directions arising from eigenfunctions.

The proof of Theorem \ref{Berger} is based on a result of independent interest: in Section \ref{Berger section} we will 
show that the Berger space admits a $5$-dimensional space of trace and divergence free Killing $2$-tensors. These are symmetric
$2$-tensors with vanishing complete symmetrisation of their covariant derivative. A special motivation for studying
Killing tensors stems from the fact that they define first integrals of the geodesic flow, i.e. functions constant on
geodesics. On the Berger space the constructed  Killing tensors turn out to be eigentensors for the Lichnerowicz Laplacian
for an eigenvalue less than the critical $2E$.

\medskip

\noindent{\bf Acknowledgements:}

The first author would like to thank P.-A. Nagy and G. Weingart for helpful discussions on the topic of this
article. He is also grateful for support by the Special Priority Program SPP 2026 "Geometry at Infinity" funded by the DFG. 
The third author acknowledges partial support by a Discovery Grant of NSERC.


%
%
\section{\bf Sasaki Einstein manifolds with $b_2 >0$}
%
%
Let $(M^{2n+1}, g, \xi, \eta, \Phi)$ be a Sasaki Einstein manifold as defined in the Introduction.
(This definition may be logically redundant but it allows us to keep technicalities to a minimum.)
It then follows that $\Phi$ is skew-symmetric, and for all tangent vectors $X$ in $M$ we have
$$\nabla_X \xi = - \Phi(X).$$
Also, the Einstein constant $E$ of $g$ is fixed to be $2n$, and this in turn fixes the Killing constant
of the Killing spinors to be $\pm \frac{1}{2}$ depending on the orientation.

The Killing field $\xi$ gives rise to a Riemannian foliation structure on $M$, as well as an
orthogonal decomposition
$$ TM = \mathscr{L} \oplus \mathscr{N} $$
where $\mathscr{L}$ is the line bundle determined by $\xi$ and $\mathscr{N}$ is the normal bundle
to the foliation. There is a transverse K\"ahler structure on $\mathscr{N}$ with a transverse
Hodge theory associated with the basic de Rham complex. Good references for this material are
 \cite{EH86} and sections 7.2, 7.3 in \cite{BG08}. In particular, the transverse almost complex
 structure is given by the endomorphism $\Phi|\mathscr{N}.$

Recall that a $k$-form $\omega$ on $M$ is called {\em basic} if $L_\xi \omega = 0 = \xi \lrcorner \, \omega$.
These properties are satisfied by $\Phi$ as well by the $K$-contact property. In the transverse Hodge
theory, the role of the K\"ahler form is played by $d\eta$, which is $\Phi$-invariant. As usual we can
define the adjoint $\Lambda$ of the wedge product with $d\eta$ and the basic elements of the kernel of $\Lambda$
are called the primitive basic forms. By Proposition 7.4.13 in \cite{BG08}, a lemma of Tachibana shows that
any harmonic $k$-form on $M$ is horizontal, and this in turn implies that it is basic, and
primitive and harmonic for the basic cohomology.

We shall also need the fact that a harmonic $2$-form on $M$ is $\Phi$-invariant, i.e., as a basic $2$-form it is
of type $(1, 1)$. This follows from the general fact that there are no nonzero basic harmonic $(0, k)$-forms,
$1 \leq k \leq n, $ which is the analogue of Bochner's theorem that on a compact complex manifold with positive first
Chern class and complex dimension $n$, there are no non-trivial holomorphic $k$-forms, $1 \leq k \leq n$.
A proof of this analogue is indicated on pp. 66-67 of \cite{VC17}  (see also \cite{GNT16}).
Note that the $\Phi$-invariance of the transverse covariant derivative and transverse curvature tensor,
which is the analogue of the K\"ahler condition, is needed in the argument.

Now let $\alpha$ be a harmonic $2$-form on our Sasaki Einstein manifold. The candidate for a destabilizing
direction is the symmetric $2$-tensor $h_{\alpha}$ defined by
\begin{equation}  \label{SE-TT}
   h_{\alpha}(X, Y) := \alpha(X, \Phi(Y))
\end{equation}
for arbitrary tangent vectors $X, Y$ on $M$.

It turns out that an efficient method to compute the action of the Lichnerowicz Laplacian on $h_{\alpha}$
is to take advantage of the canonical metric connection with skew torsion preserving the Sasakian structure
rather than using the Levi-Civita connection. We will also use the fact that the Lichnerowicz Laplacian  can be written
as $\Delta_L = \nabla^*\nabla + q(R)$, where $q(R)$ is an endomorphism on symmetric tensors  fibrewise defined by
$
q(R) = \sum (e_i \wedge e_j)_\ast \circ R (e_i \wedge e_j)_\ast
$
with an orthonormal basis $\{e_i \}$, where for any $A \in \Lambda^2 \cong \so(n)$ we denote with $A_\ast$ the natural action of $A$ on
symmetric tensors. See \cite{SWe19} for the general context of
these endomorphisms and their relation with Weitzenb\"ock formulae.

Recall that Sasakian manifolds are equipped with a canonical metric connection $\bar\nabla$ defined by the equation
$$
g(\bar\nabla_X Y, Z) \;=\; g(\nabla_X Y,  Z)  \;+\; \tfrac12 (\eta \wedge d\eta) (X, Y, Z) \ ,
$$
where the $3$-form $\eta \wedge d\eta$ is precisely the torsion of the connection.
Note that the canonical connection $\bar \nabla$  for any tangent vector $X$ can also be written as $\bar\nabla_X = \nabla_X + A_X$ with
$A_X:= - \eta(X) \, \Phi  + \xi \wedge \Phi(X)$. $\bar \nabla$ preserves the basic forms. The restricted holonomy group
of $\bar \nabla$ lies in $\U(n) \subset \SO(2n+1)$ and we have $\bar\nabla \Phi = 0$, $\bar \nabla \eta = 0$,
and $\bar \nabla d\eta = 0$.

Furthermore, the curvature $\bar R$ of $\bar\nabla$ and its action on tensors is given by
$$
\bar R_{X, Y} \;=\; R_{X, Y} \,+\, \tfrac12 \, d\eta(X,Y) \, d\eta \,-\, \Phi(X) \wedge \Phi(Y)  \,+\, \xi  \wedge (X \wedge Y)\xi,
$$
where $(X \wedge Y) \xi := g(X, \xi) Y - g(Y, \xi) X$, and we have identified vectors with covectors as usual via $g$.
Hence, we have  $\bar R_{X, Y} = R_{X, Y} \,-\, \Phi(X) \wedge \Phi(Y)  \,+\, \xi  \wedge (X \wedge Y)\xi $ \,for the action of $\bar R_{X, Y} $ on $\Phi$-invariant tensors.
As a consequence we obtain for these tensors the formula
\begin{equation}\label{diff}
q(\bar R) - q(R) \;=\; -\tfrac12 \sum (e_i \wedge e_j)_\ast \circ (\Phi(e_i) \wedge \Phi(e_j))_* \;+\; \sum (\xi \wedge e_j)_\ast \circ (\xi \wedge e_j)_\ast \ ,
\end{equation}
where $q(\bar R)$ denotes the  curvature endomorphism with respect to the connection $\bar\nabla$ and its curvature $\bar R$.
On specific spaces this difference can be further computed. The result for the present situation
is given in the following

\begin{lemma}
Let $(M^{2n+1}, g, \xi, \eta, \Phi)$ be a Sasaki Einstein manifold. Then we have
$$
q(\bar R) \,-\, q(R) \;=\;
\left\{
\begin{array}{ll}
\,\;\; 2\, \id &  \qquad \mbox{\rm on} \quad  \Lambda^2 \T M\\
-2 \, \id  &  \qquad \mbox{\rm on} \quad  \Sym^2_0 \T M \\
\end{array}
\right.
$$
for  $\Phi$-invariant and basic tensors.
\end{lemma}
\begin{proof}
We start with a remark to understand the action of the curvature term $q(R)$ on $2$-tensors (symmetric or skew-symmetric). Let $A, B \in \Lambda^2 \T \cong \End^- \T$ and let
$h$ be a $2$-tensor. Then the composed action of $A$ and $B$ on $h$ is given by
$$
(A_* B_* h) (X, Y) \;=\; h(B A X, Y) \,+\, h(A X, B Y) \,+\, h(B X, A Y) \,+\, h(X, B A Y ) \ .
$$
Hence, for computing the action of  the first summand in \eqref{diff}  on $2$-tensors we need  the following formula on tangent vectors $X$
$$
-\tfrac12 \sum (\Phi(e_i) \wedge \Phi(e_j))_\ast \,(e_i \wedge e_j)_*X \;=\; - \sum (\Phi(X) \wedge \Phi(e_j))_*  e_j \;=\; -\Phi^2(X) \;=\; X \mod \xi
$$
where we can neglect any multiples of $\xi$ since in the end we want to apply our difference formula to basic tensors. Similarly we compute the
action of the second summand  in \eqref{diff}  on vector fields. Here we obtain
$$
\sum (\xi \wedge e_j)_\ast \,  (\xi \wedge e_j)_\ast X \;=\; g(\xi, X) \, \sum (\xi \wedge e_j)_* e_j \,-\, (\xi \wedge X)_* \xi \;=\; - X \mod \xi \ .
$$
Next we have to compute
$$
- \tfrac12 \sum  [h((e_i \wedge e_j)_*X, (\Phi(e_i) \wedge \Phi(e_j))_* Y) \;+\; h( (\Phi(e_i) \wedge \Phi(e_j))_*X,  (e_i \wedge e_j)_*Y)]
\phantom{xxxxxxxx}
$$
\bea
&=&
-\sum [ h(e_j, (\Phi(X) \wedge \Phi(e_j)_*Y) \;+\; h((\Phi(Y) \wedge \Phi(e_j))_*X, e_j )]\\[1ex]
&=&
\quad \sum [g(\Phi(e_j), Y) h(e_j, \Phi(X)) \;+\; [g(\Phi(e_j), X) h(\Phi(Y), e_j)]\\[.5ex]
&=&
-2 h(\Phi(Y), \Phi(X)) \;=\; -2 h(Y, X) \;=\; \left\{
\begin{array}{ll}
\;\; 2\, h(X, Y)&  \qquad \mbox{for} \quad \;  h \in  \Lambda^2 \T M\\
-2 \, h(X, Y)  &  \qquad \mbox{for } \quad   h \in \Sym^2 \T M \\
\end{array}
\right.
\eea
Finally, we note $\sum h((\xi \wedge e_j)_*X, (\xi \wedge e_j)_*Y) = g(\xi, X) g(\xi, Y) \sum h(e_j, e_j) = 0$ on basic and trace-free $2$-tensors $h$.
Then combining the formulas above finishes the proof of the lemma.
\end{proof}

\medskip

\begin{lemma}
On tracefree, divergence-free, $\Phi$-invariant and basic $2$-tensors we have the formula:   $\bar\nabla^* \bar\nabla  \,-\, \nabla^*\nabla \,=\, - 2 \, \id$.
\end{lemma}
\begin{proof}
Let $\{e_i\}$ be a local orthonormal basis with $\nabla_{e_i} e_i = 0 =\bar  \nabla_{e_i} e_i $ at an arbitrary but
fixed point $p \in M$. Recall that the torsion of $\bar\nabla$ is skew-symmetric. Computing at the point $p$ we get
\bea
\bar\nabla^* \bar\nabla  \,-\, \nabla^*\nabla  &=& - \sum  \bar \nabla_{e_i} \bar \nabla_{e_i}  \;+\;  \sum \nabla_{e_i} \nabla_{e_i}
\;=\; \sum (\bar \nabla_{e_i} - A_{e_i} ) (\bar \nabla_{e_i} - A_{e_i} ) \;-\; \sum \bar \nabla_{e_i} \bar \nabla_{e_i} \\[1ex]
&=& \quad \sum A_{e_i }A_{e_i } \;-\; 2 \sum A_{e_i } \bar \nabla_{e_i} \ .
\eea
On $\Phi$-invariant and basic tensors the first summand  reduces to $ \sum A_{e_i }A_{e_i } = \sum  A_{e_i } (\xi \wedge \Phi(e_i))_*$.
Computing this first on tangent vectors $X$ (with  the factors in  reversed order as above) we obtain
$$
\sum   (\xi \wedge \Phi(e_i))_* A_{e_i } X \;=\; \sum   (\xi \wedge \Phi(e_i))_* (-\eta(e_i) \Phi(X) \,+\, (\xi \wedge \Phi(e_i))_*X)
\phantom{xxxxxxxx}
$$
\bea
&=&
\sum   (\xi \wedge \Phi(e_i))_* (\eta(X) \Phi(e_i) \,-\, g(\Phi(e_i), X) \, \xi )
\;=\; - \sum g(\Phi(e_i), X) \Phi(e_i) \mod \xi\\[1ex]
&=& \Phi^2(X) \mod \xi \;=\; - X \mod \xi \ .
\eea
Here we used $\Phi(\xi)= 0$ and $\xi \perp \mathrm{Im} (\Phi)$. Moreover, for tracefree, $\Phi$-invariant and basic $2$-tensors $h$ we obtain
$$
\sum h(A_{e_i} X, (\xi \wedge \Phi(e_i))_*Y )  \;+\;  h(   (\xi \wedge \Phi(e_i))_* X, A_{e_i} Y) \phantom{xxxxxxxx}\phantom{xxxxxxxxxxxx}
$$
$$
\;=\; \sum h( - \eta(e_i) \Phi(X) \,+\, g(\xi, X) \Phi(e_i), \,g(\xi, Y)\, \Phi(e_i)) \;+\; (X \leftrightarrow Y) \phantom{xxxxxxxx}
$$
$$
= \; g(\xi, X) \, g(\xi, Y) \, \sum h(\Phi(e_i), \Phi(e_i)) \;+\; (X \leftrightarrow Y)  \;=\; 0 \ .\phantom{xxxxxxxxxxxxxx}
$$
The last equation holds since $h$ is tracefree and $\Phi$-invariant. The calculation so far shows that $\sum A_{e_i }A_{e_i } h = -2h$.
Finally we have to compute the second summand in our formula for $\bar\nabla^* \bar\nabla  \,-\, \nabla^*\nabla $. Here we obtain
\bea
\sum A_{e_i } \bar \nabla_{e_i} h &=& \sum(-\eta(e_i)\Phi \,+\, (\eta \wedge \Phi(e_i))_* \bar \nabla_{e_i} h \;=\;
\sum  (\eta \wedge \Phi(e_i))_* \bar \nabla_{e_i} h
 \eea
 since $\bar \nabla_{e_i} h$ is again trace-free, $\Phi$-invariant, and basic. We conclude $\sum A_{e_i } \bar \nabla_{e_i} h = 0$ since
 we have $\sum  (\eta \wedge \Phi(e_i))_* X = \sum (g(\xi, X) \, \Phi(e_i) \,-\, g(\Phi(e_i), X) \,\xi)$.
 Indeed, clearly
 $$\sum  (\eta \wedge \Phi(e_i))_* \bar \nabla_{e_i} h(X, Y) = 0,$$
 if $g(X, \xi)=g(Y, \xi)=0$ or $X=Y=\xi$, since $\bar \nabla_{e_{i}} h$ is basic. Moreover, for $g(\xi, Y)=0$, a straightforward computation gives
 $$\sum  (\eta \wedge \Phi(e_i))_* \bar \nabla_{e_i} h(\xi, Y)=(\delta h)(\Phi(Y))=0,$$
 since $h$ is divergence-free.
\end{proof}

\medskip

We consider the two Laplace type operators $\Delta = \nabla^*\nabla + q(R)$ and $\bar \Delta = \bar\nabla^*\bar \nabla + q(\bar R)$,
where the operator $\Delta$ is just the Lichnerowicz operator on tensors. The operator $\bar \Delta$ has the important property that
it commutes with parallel bundle maps  (see \cite{SWe19}, p. 283). Combining the last two lemmas we obtain

\begin{cor}
On  $\Phi$-invariant,  divergence-free and basic tensors we have 
$$
\bar \Delta  \,-\,  \Delta \;=\;
\left\{
\begin{array}{ll}
\quad 0 &  \qquad \mbox{on} \quad  \Lambda^2 \T M\\
-4 \, \id  &  \qquad \mbox{on} \quad  \Sym^2_0 \T M \\
\end{array}
\right.
$$
\end{cor}

\medskip

Let $\alpha \in \Omega^2(M)$ be a harmonic $2$-form, which must be $\Phi$-invariant and basic. Hence we can
apply the difference formula above and obtain $\bar  \Delta \alpha = 0$.  Now the bundle of $\Phi$-invariant $2$-forms can be identified
with the bundle of symmetric $2$-tensors by the $\bar\nabla$-parallel bundle map $\alpha \mapsto h_\alpha$, where the symmetric
$2$-tensor $h_\alpha$ is given by (\ref{SE-TT}). Then, because  $\bar\Delta$ commutes with parallel bundle maps
we also have $\bar\Delta h_\alpha = 0$. Using again the difference formula above, now for the case $\Sym^2 \T M$, we obtain
$\Delta h_\alpha = 4 h_\alpha$. As the Einstein constant is $2n$, $(\Delta - 2E) h_{\alpha} = (4-4n) h_{\alpha}$ and hence
the instability condition (\ref{instability}) is satisfied when $n > 1$.

It remains to check that $h_\alpha$ is a TT-tensor. First, we see that $\tr_g (h_\alpha) = g(\alpha, d\eta) = 0$
since harmonic forms on Sasaki Einstein manifolds are primitive. Moreover, $\delta_g h_\alpha = 0$ follows by an easy
calculation from the assumption $d^* \alpha = 0$ and the fact that $\alpha$ is basic. Thus we have proved

\begin{thm} $($Theorem \ref{SE}$)$
Let  $(M^{2n+1}, g, \xi, \eta, \Phi)$ be a compact Sasaki Einstein manifold with $n>1$ and $b_2 >0$.
Then the Einstein metric $g$ is linearly unstable.  \eop
\end{thm}

%
%
\section{\bf Properties of nearly parallel $\mathrm G_2$-manifolds}  \label{facts}
%
%

In the remainder of the paper we will focus on the dimension $7$ case. In this section we will summarise
some properties of nearly parallel ${\rm G}_2$ manifolds that we shall need.

Let $(M^7, g)$ be a nearly parallel $\mathrm G_2$-manifold, i.e., a complete spin manifold with a
non-trivial real Killing spinor $\sigma$. For the moment
we do not exclude the possibility that the dimension of the space of real Killing spinors is greater than one.
We may assume that $\sigma$ has length $1$ and Killing constant $\frac{1}{2}$,
i.e., $\nabla_X \sigma = \frac{1}{2} X \cdot \sigma$.  In particular the scalar curvature is normalized as ${\rm scal}_g = 42$. Then the Killing spinor $\sigma$ determines a vector cross product by the condition
$$ P_{\sigma}(X, Y) \cdot \sigma = X \cdot Y \cdot \sigma + g(X, Y) \sigma = (X \wedge Y) \cdot \sigma, $$
and hence a $3$-form
$$ \varphi_{\sigma}(X, Y, Z) = g(P_{\sigma}(X, Y), Z).$$
For details of this construction, see e.g. \cite{FKMS97}. The stabilizers of this $3$-form belong to
the conjugacy class ${\rm G}_2 \subset \SO(7)$ and we obtain a ${\rm G}_2$-structure on $M$, which we
regard as a principal ${\rm G}_2$ bundle $Q_{\sigma}$ over $M$.
There is a unique metric connection $\bar \nabla$
on this bundle with totally skew torsion. We let $\nabla$ denote the Levi-Civita connection of $g$.

Equivalently, nearly parallel $\mathrm G_2$-manifolds can be defined as  Riemannian $7$-manifolds carrying a 3-form $\varphi$
whose stabilizer at each point is isomorphic to the group $\mathrm G_2$ and such that  $d\varphi = \lambda \ast \varphi $ for some non-zero real number $\lambda$.

Since the tensor bundles over $M$ are associated fibre bundles of $Q_{\sigma}$, we obtain orthogonal decompositions
of these bundles from the decompositions of the corresponding $\SO(7)$ representations upon restriction
to ${\rm G}_2$. The decompositions which we need are
\begin{itemize}
\item[(i)] $\Lambda^2 \T = \Lambda^2_7 \oplus \Lambda^2_{14} $
\item[(ii)] ${\rm S}^2 \T = \I \oplus {\rm S}^2_{27} $
\item[(iii)] $\Lambda^3 \T = \I \oplus \Lambda^3_{7} \oplus \Lambda^3_{27}$
\item[(iv)] the spin bundle $\D = \I \oplus \T$
\end{itemize}
where $\T$ denotes the tangent bundle of $M$, the rank of a sub-bundle is indicated by a subscript, $\I$ denotes
the ($1$-dimensional) trivial bundle, and we have identified orthogonal representations with their duals.
We further have equivalences ${\rm S}^2_{27} \cong \Lambda^3_{27},$ and $\Lambda^2_7 \cong \Lambda^3_7 \cong \T$.
Note that the trivial bundle in $\Lambda^3 \T$ is spanned by $\varphi_{\sigma}$ and that in $\D$ is spanned by
$\sigma$. We also let $\psi = \ast \varphi$, the Hodge dual of $\varphi$, which spans the trivial bundle in $\Lambda^4 \T$.

More explicitly, it is well-known (see \cite{Br87}) that
\begin{itemize}
\item[(a)] $ \Lambda^2_7 = \{ X \lrcorner \,\varphi: X \in \T \} = \{ \omega \in \Lambda^2 \T: \ast(\varphi \wedge \omega) = -2 \omega\}, $
\item[(b)] $ \Lambda^2_{14} = \{ \omega \in \Lambda^2 \T: \forall X \in \T, g(\omega, X \lrcorner \, \varphi)= 0\} =
          \{ \omega \in \Lambda^2 \T: \ast(\varphi \wedge \omega) = \omega \}, $
\item[(c)] $ \Lambda^3_{7} = \{ X \lrcorner \, \psi : X \in \T \}, $
\item[(d)] $ \Lambda^3_{27} = \{ \omega \in \Lambda^3 \T: \omega \wedge \varphi = 0 = \omega \wedge \psi \}.$
\end{itemize}

\begin{prop}  \label{harmonic-decomp}
Let $(M^7,g,\varphi)$ be a closed nearly parallel ${\rm G}_2$ manifold. Then any harmonic $2$-form is
a section of $\Lambda^2_{14} $ and any harmonic $3$-form is a section of $\Lambda^3_{27}$.
\end{prop}
\begin{proof}
The Clifford product of a harmonic form with a  Killing spinor vanishes, as was shown by O. Hijazi in \cite{Hi86}.
Moreover, for a fixed spinor $\sigma$, the map $\omega \mapsto \omega \cdot \sigma: {\rm Cl}(\R^7) \rightarrow \D$
 is a $\Spin(7)$-equivariant hence ${\rm G}_2$-equivariant homomorphism. Therefore by Schur's lemma, the components
 of a form which correspond to ${\rm G}_2$-representations which do not occur in the spin representation, e.g.,
 $\Lambda^2_{14}$ and $\Lambda^3_{27}$,  also act trivially on $\sigma$. Hence in order to finish the proof of
 the proposition it suffices to show that forms in $\Lambda^2_7, \Lambda^3_1$ and
$\Lambda^3_{7}$  act non-trivially on the Killing spinor $\sigma$ of the ${\rm G}_2$ structure.

We will need the following formula for Clifford multiplication of forms:
$$
(X\wedge \omega ) \cdot \;  =\;  X \cdot  \omega \cdot  \, +  \; (X\lrcorner \,  \omega) \cdot
$$
Here $X$ is a tangent vector, $\omega$ an arbitrary $k$-form, and $\cdot$ denotes Clifford multiplication.


Using the cross product $P$ (where we have suppressed the dependence on $\sigma$) we calculate that
$\sum P(e_i, P(e_i, X)) = -6 X$, where $\{e_i\}$ is an orthonormal basis of $\T$.
Moreover, we need the following simple formulas for $\varphi$ and its Hodge dual $\psi=\ast \varphi$
\medskip
\begin{enumerate}
\item\quad
$X \lrcorner \, \varphi \;=\; - \tfrac12 \sum e_i \wedge P(e_i, X)$
\medskip
\item \quad
$\psi=\ast \varphi \:=\; - \tfrac{1}{6} \, \sum (e_i \lrcorner \, \varphi ) \wedge (e_i \lrcorner \, \varphi ) $
\quad and \quad
$X \, \lrcorner  \ast \varphi \;=\; -\tfrac13 \sum P(e_i, X) \wedge (e_i \lrcorner \, \varphi ).$
\end{enumerate}

\medskip

\noindent
First we show that forms in $\Lambda^2_7 $ act non-trivially on the Killing spinor $\sigma$. We have
$$
(X \lrcorner \, \varphi ) \cdot \sigma \;=\;   - \tfrac12 \sum e_i \cdot  P(e_i, X) \cdot \sigma \;=\;  - \tfrac12 \sum e_i \cdot  (e_i \cdot X + g(e_i, X))\cdot \sigma \;=\; 3 X \cdot \sigma \ .
$$
Recall that the map $X \mapsto X \cdot \sigma$ is injective. Next we show the non-trivial action of $\Lambda^3_1$:
$$
\varphi \cdot \sigma \;=\;  \tfrac13 \sum e_i \wedge  (e_i \lrcorner \, \varphi)  \cdot \sigma \;=\;    \tfrac13 \sum e_i \cdot (e_i \lrcorner \, \varphi) \cdot \sigma \;=\; - 7 \, \sigma \ .
$$
Finally we have to show that $\Lambda^3_7 $ acts non-trivial on $\sigma$. Here we compute
\bea
(X \, \lrcorner \,\psi) \cdot \sigma &=& - \tfrac13 \sum (P(e_i, X ) \cdot (e_i \lrcorner \, \varphi) + P(e_i, P(e_i, X))) \cdot \sigma\\
&=&
 - \tfrac13 \sum(  -3e_i \cdot P(e_i, X) - 6 X ) \cdot \sigma \;=\; - 4 X \cdot \sigma \ .
\eea
\end{proof}

\begin{rmk}
 The same result was proved in \cite{DS20}, Thm. 3.8, Thm. 3.9. and in the $2$-form case also in \cite{BO19}, Rem. 4.
 Note that also the corresponding statement for harmonic forms on  $6$-dimensional nearly K\"ahler manifolds
 (see \cite{Fos17}, \cite{V11}) can be reproved using Killing spinors and similar arguments as those above.
\end{rmk}

\begin{rmk}
The above construction and structures clearly depend smoothly on the unit Killing spinor $\sigma$ chosen and
can be made in the Sasakian-Einstein (respectively $3$-Sasakian) case using just one of the circle's
(resp. two-sphere's) worth of unit Killing spinors. By contrast, the harmonic forms associated to the metric $g$
do not depend on the structures determined by $\sigma$.
\end{rmk}

As already mentioned in the introduction, there are three classes of nearly parallel $\mathrm G_2$-manifolds: $3$-Sasakian, Sasaki Einstein  and the proper nearly parallel $\mathrm G_2$-manifolds. The homogeneous examples were classified in  \cite{FKMS97}. In the proper case we only have the squashed $7$-sphere $S^7_{sq}$, the Aloff-Wallach spaces $N_{k,l}$ and the Berger space $\SO(5)/\SO(3)_{irr}$. The only other
known class of  proper nearly parallel $\mathrm G_2$-structures are given by the second Einstein metric in the canonical variation of the $3$-Sasaki metrics in dimension $7$. In particular  $S^7_{sq}$ and $N_{1,1}$ belong to this class.  For the purpose of our article we note that it is well-known that both Einstein metrics in the canonical variation are unstable.
(See \cite{Be87}, 14.85 together with Fig. 9.72.)

%
%
\section{\bf Nearly parallel $G_2$ manifolds with $b_3 >0$}
%
%
In this section we prove
\begin{thm}  \label{dim7-b3} $($Theorem \ref{b3} $)$
Let $(M^7, g, \varphi)$ be a nearly parallel ${\rm G}_2$ manifold admitting a non-trivial harmonic $3$-form,
i.e. with $b_3  > 0$. Then the Einstein metric $g$ is linearly unstable.
\end{thm}
\begin{proof}
Let $\beta$ be a harmonic  $3$-form on $M$. By Proposition \ref{harmonic-decomp}  it is a section of  $\Lambda^3_{27}$.
Consider the tracefree symmetric $2$-tensor $h := j(\beta)$ defined by the identification map
$j:  \Lambda^3_{27}  \rightarrow {\rm S}^2_0 \T ,$ which was first studied by Bryant \cite{Br05}.
Since $\beta $ is  harmonic and in particular closed, Proposition 6.1 in \cite{AS12} immediately implies $\Delta_L h = \tfrac{\tau_0^2}{4} h =   \tfrac{2\, \scal_g}{21} h = 4h$,
since we have normalized the scalar curvature to $\scal_g=42$, which corresponds to choosing $\tau_0 =4$. Hence, $h$ is a $\Delta_L$-eigentensor for the eigenvalue $4< 2E = 12$.
It remains to show that $h$ is also divergence-free and thus a TT-tensor, since it is trace-free by definition.
Here an easy calculation, which is essentially also  contained in \cite{AS12},
shows that $\delta j(\beta) = -2 P(d^* \beta)$ holds for any section $\beta$ of  $\Lambda^3_{27}$, where $P$ is the
vector cross product introduced in the previous section. Since a harmonic form $\beta$ is also co-closed we conclude
that $h= j(\beta)$ is divergence free.
\end{proof}

\medskip

Note that K. Galicki and S. Salamon showed that the odd Betti numbers of a compact $3$-Sasakian manifold
must vanish (see \cite{GaS96}, Theorem A). So the above result does not apply when the nearly parallel ${\rm G}_2$-manifold is $3$-Sasakian. As well, there is no known example of a proper nearly parallel $G_2$-manifold with $b_3>0$.
However, there are some examples of Sasaki Einstein $7$-manifolds admitting non-trivial harmonic $3$-forms (see Introduction).

%
%
\section{\bf The Berger space}  \label{Berger section}
%
%
The aim of this section is to show that the nearly parallel ${\rm G}_2$-metric on the Berger space
$M^7 = \SO(5)/\SO(3)_{irr}$ is linearly unstable. Since the Berger space has no non-trivial harmonic forms
we cannot use the arguments of the previous sections. Instead we will prove the instability  by showing that
there are symmetric Killing tensors which are eigentensors of the Lichnerowicz Laplacian with eigenvalue  less than $2E$.
We start with recalling a few facts on Killing tensors, referring to \cite{HMS17} for further details.

Recall that a symmetric tensor $h\in \Gamma(\Sym^p \T M)$ is called a {\it Killing tensor} if the complete
symmetrization of $\nabla h $ vanishes, i.e., if $(\nabla_X h)(X, \ldots, X) = 0$ holds for all tangent vectors $X$.
Let $d : \Gamma(\Sym^p \T M) \rightarrow \Gamma(\Sym^{p+1} \T M)$
be the differential  operator defined by $d  h = \sum_i e_i \cdot \nabla_{e_i} h$, where $\{e_i\}$ is a local orthonormal
basis and $\cdot$ now denotes the symmetric product. Then the Killing condition for the symmetric tensor $h$ is equivalent to $d h = 0$.
For a trace-free symmetric Killing tensor $h\in \Gamma(\Sym^p_0 \T M)$ it is easy to check that $h$ is also divergence-free, i.e.,
trace-free Killing tensors are TT-tensors (see Corollary 3.10 in  \cite{HMS17}).

On compact Riemannian manifolds Killing tensors can be characterized using the Lichnerowicz Laplacian $\Delta_L$ acting on
symmetric tensors. Recall that $\Delta_L$ can be written as $\Delta_L = \nabla^*\nabla + q(R)$. Then an easy calculation shows
$\Delta_L  h \ge 2 q(R) h $ on divergence-free symmetric tensors, with equality exactly for divergence-free Killing tensors $h$, i.e., $h\in \Gamma(\Sym^2_0 \T M)$
is a Killing tensor if and only if $\Delta_L h = 2 q(R) h$  (see Proposition 6.2 in \cite{HMS17}).


The isotropy representation of the Berger space  $ M^7 = \SO(5)/\SO(3)_{irr}$ is the unique  $7$-dimensional
irreducible representation of $\SO(3)$. It also defines an embedding of $\SO(3)$ into ${\rm G}_2$ and thus a ${\rm G}_2$-structure on
$M^7$  which turns out to be a  proper nearly parallel  ${\rm G}_2$-structure (see \cite{Br87}, p. 567).

Replacing the groups $\SO(5)$ and $\SO(3)$ by their double covers, we can realize the Berger space also as
$M^7= \Sp(2)/\Sp(1)_{irr}$ where $\Sp(1)$ is embedded by the unique $4$-dimensional irreducible representation.
We write $\sp(2) = \sp(1) \oplus \m$, where $\m$ is the orthogonal complement of $\sp(1)$ with respect to the Killing
form $B$ of $\sp(2)$. As usual, $\m$ is identified with the tangent space at the identity coset and, as mentioned above,
is the irreducible   $7$-dimensional  representation of $\Sp(1)$. Recall that the irreducible complex  representations of
$\Sp(1)$ can be written as the symmetric powers $\Sym^k E$, where $E = \C^2$ is the
standard representation of  $\Sp(1)$.  In particular, we have $\m^{\scriptsize\C} := \m \otimes \C = \Sym^6 E$. We will need the
following decomposition into irreducible summands:
\begin{equation}\label{deco1}
\Sym^2_0 \, \m^{\scriptsize\C} \cong \Lambda^3_{27} \m^{\scriptsize\C} \cong \Sym^4 E \oplus \Sym^8 E  \oplus \Sym^{12} E \ .
\end{equation}

The Peter-Weyl theorem and the Frobenius reciprocity now imply the following
decomposition  into irreducible summands of the left-regular representation of $\Sp(2)$   on sections of the vector bundle $\Sym^2_0  \T M^{\scriptsize\C}$
\begin{equation}\label{deco2}
\Gamma(\Sym^2_0 \,  \T M^{\scriptsize\C}) \cong \overline \bigoplus_{k,l} \, V(k,l) \otimes \Hom_{\Sp(1)} (V(k,l), \, \Sym^2_0\, \m^{\scriptsize\C})
\end{equation}
where the sum goes over all pairs of integer $(k, l)$ with $k \ge l \ge 0$. Here  $V(k,l)$ is the irreducible $\Sp(2)$-representation
with highest weight $\gamma = (k,l)$, where $k$ corresponds to the short simple root,  and it is easy to compute that the $\sp(2)$-Casimir operator (with respect to the Killing form) acts on the representation space $V(k,l)$ as $-\tfrac{1}{12}(4k + k^2 + 2l + l^2)\id$. In particular, $V(2,0)$ is the adjoint representation of
$\sp(2)$ and the Casimir eigenvalue is $-1$, as it should be.

Interesting for us will be the representation $V(1,1)$ with Casimir eigenvalue $-\tfrac23$. It is easy to check that $V(1,1)$ is $5$-dimensional
and that $V(1,1) = \Sym^4 E$ considered as an $\Sp(1)$-representation. Moreover, from $\m^{\scriptsize\C} = \Sym^6 E$ and the decomposition \eqref{deco1}
we conclude
\begin{equation}\label{hom}
\dim \Hom_{\Sp(1)}(V(1,1), \, \Sym^2_0 \, \m^{\scriptsize\C}) = 1
\quad \mbox{and} \quad
\Hom_{\Sp(1)}(V(1,1), \,\m^{\scriptsize\C} ) = \{0\} \ .
\end{equation}
Using the program LiE for the decomposition of $\Sym^3 \Sym^6E$ as an $\Sp(1)$-representation it is also easy to check that $\Hom_{\Sp(1)}(V(1,1), \,\Sym^3 \m^{\scriptsize\C}) = \{0\}$.
This has the following important consequence:

\begin{lemma}
The space $V(1,1) \subset \Gamma(\Sym^2_0 \,  \T M^{\scriptsize \C} )$ consists of tracefree Killing tensors.
\end{lemma}
\begin{proof}
Killing $2$-tensors are by definition symmetric tensors in the kernel of the differential operator $d :\Gamma(\Sym^2 \, \T M) \rightarrow \Gamma(\Sym^3 \, \T M)$
introduced above.  In our situation the operator $d$ is an $\Sp(2)$-invariant differential operator.
 Hence, with respect to the decomposition \eqref{deco2} it restricts to an invariant map
$$
V(1,1)\otimes  \Hom_{\Sp(1)}(V(1,1), \, \Sym^2_0 \, \m^{\scriptsize\C}) \longrightarrow V(1,1) \otimes \Hom_{\Sp(1)}(V(1,1), \,\Sym^3 \m^{\scriptsize\C}) \ ,
$$
which has to be zero since the space on right side vanishes. A similar argument shows that also the divergence of elements in $V(1,1)$ has to be zero.
However, since elements in $V(1,1)$ are trace-free by definition this was already clear from   a remark above.
\end{proof}

Since we  know that $\Delta_L$ acts by the curvature term $2 q(R)$ on divergence-free Killing tensors it
remains to compute the action of $q(R)$ on the space $V(1,1)$. For doing this we will use the comparison formula
(5.33) in \cite{AS12}. Let $(M^7, g, \varphi)$ be a nearly parallel $\rm G_2$ manifold, with
$d \varphi = \tau_0 \ast \varphi$ and let $\bar R$ denote the curvature of the canonical ${\rm G}_2$ connection
 $\bar \nabla$ then this formula states
\begin{equation}\label{qr}
q(R)  \;=\; q(\bar R) \; +\; 3 (\tfrac{\tau_0}{12})^2 \, \Cas^{\so(7)} \,-\;4 (\tfrac{\tau_0}{12})^2 \, S
\end{equation}
where $ \Cas^{\so(7)} $ is the  Casimir operator of $\so(7)$ which acts as $-k(7-k)\id$ on forms in $\Lambda^k \T$   and as $-14\,\id$ on $\Sym^2_0 \T$.
The ${\rm G}_2$ invariant endomorphism  $S$ is defined as $S = \sum P_{e_i} \circ P_{e_i}$, where $P_X$ is the skew-symmetric endomorphism
$X \lrcorner \, \varphi$. Then a straight forward computation on explicit elements shows that $S = - 6\, \id $ on $\T$ and $S = -14 \,\id$ on $\Sym^2_0 \T$.

\medskip

As an example we consider the difference of $q(R)$ and $q(\bar R)$ on $\T M$. Of course we already know that
$q(R) = \Ric = \tfrac{\scal}{7} = \tfrac{3\tau^2_0}{8}$. Hence formula \eqref{qr} and the explicit values for the Casimir operator and the endomorphism
$S$ yield  $ q(\bar R) = q(R) - \tfrac{\tau^2_0}{24} = \tfrac{\tau^2_0}{3}$.

\medskip

Returning to the Berger space,  recall that the induced Einstein metric $g$ defined by the ${\rm G}_2$ structure is
given by a multiple of the Killing form $B$ of $\sp(2)$ restricted to $\m$. Moreover, the canonical homogeneous connection
coincides with the canonical ${\rm G}_2$ connection $\bar \nabla$.  For our calculation we will use the normalization $g= -B$ for
which we have  $\tau^2_0 = \tfrac{6}{5}$ and $\scal = \tfrac{63}{20}$ (see \cite{AS12}, Lemma 7.1). In this situation
we know that the endomorphism $q(\bar R)$ acts as  $ -\Cas^{\sp(1)}$ (see \cite{AS12}, Lemma 7.2).
The Casimir operator of $\sp(1)$  with respect to the Killing form  acts on  $\Sym^k E$ as  $-k(k+2)\,\id$.
Thus $q(\bar R)$ acts on the bundle defined by the $\Sp(1)$ representation $\Sym^k E$ as $c k(k+ { 2})\,\id$ for some
 positive constant $c$, which can be determined from the case $k=6$. Indeed here we have
$\Sym^6E = \m^{\scriptsize \C}$ and we already calculated  that $q(\bar R) = \tfrac{\tau^2_0}{3} = \tfrac25$ on $\m^{ \scriptsize \C}$. It follows $c= \tfrac{1}{120}$. In particular we conclude that $q(\bar R) = \tfrac15 \,\id$ on $\Sym^4 E$ and  \eqref{qr} in our normalization
implies that $q(R) = \tfrac{19}{60}\,\id$ on the space $V(1,1)$. Hence, for any Killing tensor
 $h \in V(1,1) \subset \Gamma(\Sym^2_0 \, \m^{\scriptsize \C})$ we have $\Delta_L h = 2q(R)\, h = \tfrac{19}{30}\, h$.
But $\tfrac{19}{30} <  2E = \tfrac{3\tau^2_0}{4} = \tfrac{9}{10} $, so the instability condition (\ref{instability}) holds.

\medskip

We have therefore proved the following

\begin{thm}  $($Theorem \ref{Berger}$)$
Let $M^7 = \SO(5)/\SO(3)_{irr}$ be the Berger space equipped with its homogenous proper nearly parallel ${\rm G}_2$ structure. Then the induced Einstein metric is linearly unstable. Moreover, $M^7$ admits a $5$-dimensional space of divergence and
tracefree Killing tensors. \eop
\end{thm}

\begin{rmk}  \label{eigenfunction}
The reader might wonder if $\nu$-linear instability for the Berger space can be shown by looking at the spectrum
of the Laplacian on functions. It turns out that this is not possible because computations similar to the above
show that the smallest nonzero eigenvalue is larger than $2E$, and the corresponding eigenspace is the $35$-dimensional
irreducible module $V(4, 0)$.
\end{rmk}


\end{document}